\documentclass{article}

\usepackage[utf8]{inputenc}

\usepackage[top = 25mm, bottom = 25mm, left = 30mm, right = 20mm, a4paper]{geometry}

\usepackage{hyperref}

\usepackage{amssymb}
\usepackage{amsfonts}
\usepackage{amsmath}
\usepackage{amsthm}

\usepackage{enumerate}

\usepackage[x11names]{xcolor}

\newtheorem*{abstract*}{Abstract}
\newtheorem{theorem}{Theorem}[section]
\newtheorem{assumption}{Assumption}

\newtheorem{example}[theorem]{Example}

\newtheorem{lemma}[theorem]{Lemma}

\newtheorem{proposition}[theorem]{Proposition}
\newtheorem{remark}[theorem]{Remark}

\usepackage[none]{hyphenat}

\usepackage{soul}

\begin{document}

\sloppy

\title{On a version of hybrid existence result for a system of nonlinear equations}

\author{Micha{\l} Be{\l}dzi{\'n}ski, Marek Galewski and Igor Kossowski}
		
\maketitle

\begin{center}
    Institute of Mathematics, Lodz University of Technology,\\
    Lodz, Poland.
\end{center}

Corresponding author e-mail: marek.galewski@p.lodz.pl

\abstract{Combining monotonicity theory related to the parametric version of the Browder-Minty Theorem with fixed point arguments we obtain hybrid existence results for a system of two operator equations. Applications are given to a system of boundary value problems with mixed nonlocal and Dirichlet conditions.}
 
\section{Introduction}

Using the general result about the continuous dependence on parameters for
fixed points obtained via the Banach Contraction Principle, Avramescu in 
\cite{Avramescu} proved what follows:

\begin{theorem}
	\label{TheoOFAvramescu}Let $(D_{1},d)$ be a complete metric space, $D_{2}$ a
	closed convex subset of a normed space $(Y,\|\cdot\|)$, and
	let $N_{i}:D_{1}\times D_{2}\longrightarrow D_{i}$, $i=1,2$ be continuous
	mappings. Assume that the following conditions are satisfied:\newline
	(a) There is a constant $L\in \lbrack 0,1)$, such that: 
	\begin{equation*}
		d\big(N_{1}(x,y),N_{1}(\overline{x},y)\big)\leq Ld(x,\overline{x})\text{ for all }x,\overline{x}%
		\in D_{1}\text{ }and\text{ }y\in D_{2};
	\end{equation*}%
	(b) $N_{2}(D_{1}\times D_{2})$ is a relatively compact subset of $Y$.\newline
	Then, there exists $(x,y)\in D_{1}\times D_{2}$ with:%
	\begin{equation*}
		N_{1}(x,y)=x,\text{ }N_{2}(x,y)=y.
	\end{equation*}
\end{theorem}

The above mentioned fixed point theorem has recently been revisited in \cite{PrecupFixed} by imposing partial variational structure on the second mapping. The existence result obtained in \cite{PrecupFixed} instead of the Schauder Theorem employs in the proof the Ekeland Variational Principle. Hence a type of hybrid existence result has been obtained and further used for proving the existence of periodic solutions to a second order semi-linear system of ODEs. Relations to the Krasnoselskii Theorem are also explained. We observe here that a constant $L$ in the above can be replaced with a continuous function having values in $\lbrack 0,1)$ in which case Theorem \ref{TheoOFAvramescu} remains true with similar proof as recalled in  \cite{PrecupFixed}. Such remark is also valid if one considers main result from \cite{PrecupFixed}. We do not provide proof since these follow like in the sources mentioned. Now our slightly modified version reads:
\begin{theorem}
    Let $(D,d)$ be a complete metric space, $U$ an open subset of Hilbert space $Y$ identified with its dual. $N\colon D \times \overline{U}\longrightarrow D$ is continuous and $E \colon D \times \overline{U}\longrightarrow \mathbb{R}$ is a functional such that $E$ and $E'_y$ are continuous on $D \times \overline{U}$, where $E'_y$ denotes derivative of $E$ along second coordinate. Assume that the following conditions are satisfied:
    \begin{enumerate}
        \item There exists a continuous function $L\colon \overline{U} \longrightarrow [0,1)$ such that
        \begin{equation*}
            d\big( N(x,y), N (\overline{x},y)\big)\leq L(y) d(x,\overline{x})\quad \text{for all }x,\overline{x}\in D \text{ and }y\in \overline{U};
        \end{equation*}
        \item For every $x\in D$, $E(x, \cdot)$ is bounded from below on $\overline{U}$ and there is $a > 0$ with:
        \begin{equation*}
            \inf_{\partial U}E(x,\cdot) - \inf_{\overline{U}} E(x,\cdot) \geq a \quad \text{for every }x\in D;
        \end{equation*}
        \item there are two constants $R_0,\gamma > 0$ such that
        \begin{equation*}
            E(x,y) - \inf_{\overline{U}} E(x,\cdot) \geq \gamma \quad \text{for all }x\in D\text{ and }y \in \overline{U}\text{ with }\|y\| > R_0.
        \end{equation*}
        Moreover, the mapping $(N, I_Y - E'_y)$ is a Perov contraction on $D\times \overline{U}$ with respect to vector-valued metric $\widehat{d}\colon D\times Y \longrightarrow \mathbb{R}^2$,
        \begin{equation*}
            \widehat{d}\big((x_1,y_1),(x_2,y_2)\big) = \left(d(x_1,x_2), \|y_1 - y_2\|\right).
        \end{equation*}
    \end{enumerate}
    Then there exist $(x,y) \in D \times \overline{U}$ with 
    \begin{equation*}
        N(x,y) = x,\quad E'_y(x,y) = 0,\quad E(x,y) = \inf_{\overline{U}}E(x,\cdot).
    \end{equation*}
\end{theorem}

Our aim in this work is to look at hybrid existence theorems from some other point of view, namely:
a) we get rid of the usage of the Banach Contraction Principle for the first equation;
b) we do not impose the uniqueness of a solution to the first equation in the system;
c) we still retain some fixed point technique for the second one.
Thus what we obtain is a type of hybrid existence result derived for systems of nonlinear equations but having also some connections to the Krasnoselskii Fixed Point Theorem. We start investigations with replacing the usage of a parameter dependent Banach Fixed Point Theorem with a parameter dependent monotonicity method based on a recent parametric version of the Browder-Minty Theorem. Since the mentioned tool is based on the coercivity of an underlying operator, we are in position to coin it further with the Schauder Fixed Point Theorem as is done in \cite{Avramescu}. In the proofs the theory of monotone operators is used (we follow \cite{galewskiBOOKS}, \cite{Motreanu2003} and \cite{Papageorgiou2019} for some background). Some relations to existing literature are mingled further on in the text when we compare our results with what is known.

The paper is organized as follows. Firstly we provide some necessary background on the monotonicity and variational tools which we employ. Next we proceed to our hybrid existence result. Some version of the Krasonselskii Fixed Point Theorem is then given which contains the usage of the Strongly Monotone Principle instead of the Banach Contraction Theorem. The applications to nonlinear systems in which the first equation is not semilinear are given in the last section. We consider the system of ordinary differential equations such that the second equation driven by the $q$-Laplacian is subject to nonlocal boundary conditions and therefore may not be considered by the monotone or the variational approach. On the other hand the first equation corresponds to the perturbed $p$-Laplacian Dirichlet problem driven by the Leray-Lions type operator whose consideration is much more difficult with fixed point approaches than with the monotonicity one as we do in this submission. An example is given at the end showing the applicability of our assumptions.

\section{Preliminaries}

\subsection{Monotone operators}

Let $X$ be a real, reflexive and separable Banach space. Recall that $X$ is called \emph{strictly convex} if for all distinct elements $u,w$ of a unit sphere we have $\|u + w\| < 2$. Consider an operator $A\colon X\longrightarrow X^*$. We say that $A$ is \emph{radially continuous} if for all $u,w\in X$ we have 
\begin{equation*}
	\tau_n \to \tau \text{ in }[0,1]\implies \langle A(u + \tau_n w),w\rangle \to \langle A(u + \tau w),w \rangle \text{ in }\mathbb{R}
\end{equation*}
If $u_n \to u$ in $X$ implies $A(u_n)\rightharpoonup A(u)$ in $X^*$, we say that $A$ is \emph{demicontinuous}. Both notions of continuity coincide (see \cite{galewskiBOOKS}) when $A$ is \emph{monotone}, that is when
\begin{equation*}
	\langle A(u) - A(w), u - w\rangle \geq 0\quad \text{for all }u,w\in X.
\end{equation*}
In case when the above inequality is strict for all distinct $u,w\in X$ we say that $A$ is \emph{strictly monotone}. Moreover, operator $A$ is called \emph{bounded} if it maps norm bounded subsets of $X$ into norm bounded subsets of $X^*$, and \emph{coercive} if there exists a coercive function $\gamma\colon [0,\infty) \longrightarrow \mathbb{R}$ such that 
\begin{equation*}
	\langle A(u),u\rangle \geq \gamma(\|u\|)\|u\|\quad \text{for every }u\in X.
\end{equation*}
We equip $X^*$ with a standard norm and recall that whenever $X^*$ is strictly convex, for every $u\in X$ there is a unique $J(u) \in X^*$ satisfying 
\begin{equation*}
	\|u\|^2 = \|J(u)\|_*^2 = \langle J(u), u\rangle.
\end{equation*}
The mapping $X\ni u \longmapsto J(u) \in X^*$ is called \emph{the duality mapping}. Using basic properties of the duality mapping we can show the following result.

\begin{proposition}
	\label{LemmaDuality}
	Assume that $X$ is a real and reflexive Banach space. Then there exists a demicontinuous, bounded, coercive and strictly monotone operator $J\colon X\longrightarrow X^*$ satisfying $J(0) = 0$.
\end{proposition}
\begin{proof}
    By \cite{Linderstrauss} there exists a norm $\|\cdot\|_1$, equivalent with $\|\cdot\|$, such that ${(X^*,\|\cdot\|_{1,*})}$ is strictly convex. Let $J$ denote the duality mapping on $(X,\|\cdot\|_1)$. Therefore, by \cite{galewskiBOOKS}, $J$ is demicontinuous, bounded, coercive and strictly monotone on $(X, \|\cdot\|_1)$. Since those notions are clearly not violated by taking equivalent norms, $J$ satisfies assumptions of lemma.
\end{proof}

When $X$ is a Hilbert space, $J$ is a linear isomorphism between $X$ and $X^*$, whose inverse is called \emph{the Riesz Operator}.

\subsection{Sobolev Spaces}

For convenience of the Reader we recall definitions and basic properties of Sobolev spaces on bounded interval $[0,1]$ following \cite{galewskiBOOKS} and \cite{MawhinProblemesDeDirichlet}. Take $p \in (1,\infty)$. Function $u\in L^p(0,1)$ is called \emph{weakly differentiable} if there exists $v\in L^p(0,1)$, called a \emph{weak derivative of $u$}, such that 
\begin{equation*}
    \int_0^1 u(t) \phi'(t) dt = -\int_0^1 v(t) \phi(t) dt\quad \text{for every }\phi \in C^\infty_c(0,1).
\end{equation*}
Here $C^\infty_c(0,1)$ denotes the space of all smooth and compactly supported functions $\phi\colon [0,1]\longrightarrow\mathbb{R}$. That is $\phi \in C^\infty_c(0,1)$ if and only if $\phi$ is of class $C^\infty$ and $\overline{\{t\in [0,1] : \phi(t) \neq 0\}}\subset (0,1)$. We put $\dot{u} := v$ and 
\begin{equation*}
    W^{1,p}(0,1) := \left\{u\in L^p(0,1) : \dot{u}\text{ exists and }\dot{u}\in L^p(0,1)\right\}.
\end{equation*}
We endow $W^{1,p}(0,1)$ with a standard norm
\begin{equation*}
    \|u\|_{W^{1,p}} := \left(\int_0^1 |\dot{u}(t)|^p dt + \int_0^1 |u(t)|^p dt\right)^\frac{1}{p}
\end{equation*}
and let $W^{1,p}_0(0,1)$ be a closure of $C^\infty_0(0,1)$ with respect to the norm $\|\cdot\|_{W^{1,p}_0}$. Clearly $W^{1,p}_0(0,1)$ is a closed subspace of $W^{1,p}(0,1)$. Moreover, there exists a positive number $\lambda_p > 0$, called \emph{the Poincar{\'e} constant}, characterised by 
\begin{equation*}
    \lambda_p = \inf_{u\in W^{1,p}_0(0,1)\setminus \{0\}} \left\{\displaystyle\int_0^1 |\dot{u}(t)|^p dt : \displaystyle\int_0^1 |u(t)|^p dt = 1\right\}.
\end{equation*}
Therefore a norm 
\begin{equation*}
    \|u\|_p := \left(\int_0^1 |\dot{u}(t)|^p dt\right)^\frac{1}{p}
\end{equation*}
is equivalent with $\|\cdot\|_{W^{1,p}}$ on $W^{1,p}_0(0,1)$. An embedding $W^{1,p}(0,1) \hookrightarrow C([0,1])$ is compact, see \cite{MawhinProblemesDeDirichlet} for details. Moreover an inequality
\begin{equation*}
    \|u\|_\infty \leq \|u\|_p,
\end{equation*}
called \emph{the Sobolev embedding}, holds for all $u\in W^{1,p}_0(0,1)$. Here $\|u\|_\infty := \sup_{0\leq t\leq 1}|u(t)|$. 
%Now we are in a position to introduce a notion of weak solvability of equation
%\begin{equation}
    %\label{EquationIntroduction}
   % \left\{
    %\begin{array}{ll}
    %    -\frac{d}{dt}\left(|\dot{u}(t)|^{p-2}\dot{u}(t)\right) = f(t) & \text{for }t\in (0,1), \\
    %    u(0) = u(1) = 0. & 
  %  \end{array}
  %  \right.
%\end{equation}
%Function $u\in W^{1,p}_0(0,1)$ is called \emph{a weak solution} to \eqref{EquationIntroduction} if 
%\begin{equation*}
 %   \int_0^1 |\dot{u}(t)|^{p-2}\dot{v}(t)\dot{u}(t) dt = \int_0^1 f(t) v(t) dt\quad \text{for all }v\in W^{1,p}_0(0,1).
%\end{equation*}

\section{Main results}

In this section we are concerned with studying the existence of solution to

\begin{equation}
	\label{MainSystem}
	\left\{
	\begin{array}{l}
		F(u,v) = 0, \\
		G(u,v) = v
	\end{array}
	\right.
\end{equation}

In the most of the approaches which follow \cite{Avramescu} a uniqueness of solution to first equation plays a crucial role and as mentioned it is obtained by the Banach Contraction Principle. Then the Schauder Fixed Point Theorem, the Schaefer or the Ekeland Variational Principle is used to provide the solvability result for system of equations. In the approach which we propose, using some perturbation technique based on the existence of a suitable duality mapping and the usage of the parametric Browder-Minty Theorem, we will omit the uniqueness assumption along with fixed point approach concerning the solvability of the first equation.

We provide main assumptions for this section. Notice that in Assumption \ref{AssumptionH2} we do not require the superlinear growth of $\gamma$.

\begin{assumption}
	\label{AssumptionH1}
    $Y$ is a normed space, $X$ is a real, reflexive and separable Banach space.
\end{assumption}
\begin{assumption}
	\label{AssumptionH2}
	Let $F\colon X\times Y \longrightarrow X^*$. For every fixed $v\in Y$ operator $F(\cdot,v)$ is monotone and radially continuous, while $F(u,\cdot)$ is continuous for every $u\in X$. Moreover there exists a function $\gamma\colon [0,\infty)^2\longrightarrow \mathbb{R}$ such that
	\begin{equation*}
		\langle F(u,v),u\rangle \geq \gamma(\|u\|,\|v\|)\quad \text{for all }u\in X\text{ and }v\in Y
	\end{equation*}
	and that 
	\begin{equation*}
		\lim_{x\to \infty}\gamma(x,y) = \infty
	\end{equation*}
	uniformly with respect to $y$ on every bounded interval.
\end{assumption}
\begin{assumption}
	\label{AssumptionH3}
	Operator $G\colon X\times Y\longrightarrow Y$ is compact, that is it maps bounded subsets of $X\times Y$ onto relatively compact subsets of $Y$, and continuous in the following way
	\begin{equation}
        \label{ContinuityCondition}
		\left.
		\begin{array}{r}
			u_n\rightharpoonup u \text{ in }X \\
			v_n \to v \text{ in }Y
		\end{array}
		\right\}\implies G(u_n,v_n) \to G(u,v)\text{ in }Y.
	\end{equation}
\end{assumption}
Let us recall after \cite{BeldzGalKos} the parametric version of the Browder-Minty theorem which we will use as an auxiliary tool for the perturbed problem in the proof of the main result.

\begin{proposition}
	\label{PropostionPreviouPaper}
	Assume that $Y$ is a metric space, $X$ is a reflexive Banach space. If $A\colon X\times Y \longrightarrow X^*$ is an operator such that:
	\begin{itemize}
		\item $A(\cdot,y)\colon X \longrightarrow X^*$ is radially continuous and strictly monotone for all $y\in Y$,
            \item $A(u,\cdot)\colon Y \longrightarrow X^*$ is continuous for every $u\in X$;
		\item for every $y_0\in Y$ there exists an open neighbourhood $V$ of $y_0$ and a coercive function $\textcolor{red}{\rho}\colon [0,\infty) \longrightarrow \mathbb{R}$ such that 
		\begin{equation*}
			\langle A(u,y), u\rangle \geq \textcolor{red}{\rho}(\|u\|)\|u\|\quad \text{for all }y\in V\text{ and }u\in X,
		\end{equation*}
	\end{itemize}
	then for every $y\in Y$ there exists a unique $u_y$ such that $A(u_y,y) = 0$. Moreover $y_n\to y$ in $Y$ implies $u_{y_n} \rightharpoonup u_y$ in $X$.
\end{proposition}
Now we are in a position to give a first result which has rather a theoretical character. 

\begin{theorem}
	\label{TheoremNonLinearEigenvalue}
	Let Assumptions \ref{AssumptionH1}, \ref{AssumptionH2} and \ref{AssumptionH3} hold. Assume that there exists a bounded and convex set $C\subset Y$ such that $G(X\times C) \subset C$. Then there exists at least one solution to \eqref{MainSystem}.
\end{theorem}

\begin{proof}
    Let $J\colon X \longrightarrow X^*$ be an operator satisfying assertion of Proposition \ref{LemmaDuality}. For fixed $n\in \mathbb{N}$ we define operator $F_n\colon X\times Y \longrightarrow X^*$ by the formula 
	\begin{equation}
		\label{FnProof}
		F_n(u,v) = F(u,v) + \tfrac{1}{n}J(u).
	\end{equation}
	Then $F_n$ satisfies all assumptions of Proposition \ref{PropostionPreviouPaper}. It follows by the following inequality 
    \begin{equation*}
        \langle F_n(u,v),u\rangle = \langle F(u,v),u\rangle + \tfrac{1}{n}\langle J(u), u \rangle \geq \gamma(\|u\|,\|v\|) + \tfrac{1}{n}\|u\|^2\quad \text{for all }u\in X\text{ and }v\in Y.
    \end{equation*}
    Consequently for each $v_0\in Y$ is is sufficient to take $V = \{v\in X : \|v - v_0\| < 1$ and $\rho(x,y) = \frac{\gamma(x,y)}{x + 1} + \frac{1}{n}x$. Therefore for each $v\in Y$ there is a unique $S_n(v)$ satisfying $F_n(S_n(v),v) = 0$. Moreover $S_n\colon  Y \longrightarrow X$ is continuous in the following sense
	\begin{equation*}
		v_m \to v \text{ in }Y\implies S_n(v_m)\rightharpoonup S_n(v)\text{ in }X.
	\end{equation*}
	Notice that we have $\gamma(\|S_n(v)\|,\|v\|)\leq 0$ for every $v\in Y$ and all $n\in\mathbb{N}$. Therefore the set ${K := \{S_n(v) : n\in \mathbb{N}, v \in C\}}$ is bounded. Define $G_n\colon C\longrightarrow C$ by $G_n(v) = G(S_n(v),v)$. Applying the Schauder Fixed Point Theorem we obtain that there exists at least one fixed point $v_n$ of $G_n$. Denote $u_n := S_n(v_n)$. Then
	\begin{equation}
		\label{AuxiliarySystem}
		\left\{
		\begin{array}{l}
			F(u_n,v_n) = -\frac{1}{n}J(u_n), \\
			G(u_n,v_n) = v_n.
		\end{array}
		\right.
	\end{equation}
    Since ${(v_n)_n = \big( G(u_n,v_n)\big)_n}$ we have that it is compact by Assumption \ref{AssumptionH3}. Moreover (again up to subsequence) $u_n \rightharpoonup u$ since $(u_n) \subset K$. Fix $w\in X$ and let ${\nu_t = u - tw}$, $t > 0$. Then
	\begin{equation*}
		\langle F(u_n,v_n) - F(\nu_t, v_n), u_n - u\rangle + t\langle F(u_n,v_n), w\rangle > t\langle F(\nu_t,v_n), w\rangle .
	\end{equation*}
	Since $F(u_n,v_n)\to 0$, $F(\nu_t,v_n) \to F(\nu_t, v)$ and $u_n \rightharpoonup u$,
	\begin{equation*}
		\lim_{n\to \infty} \langle F(u_n,v_n) - F(\nu_t, v_n), u_n - u\rangle = 0.
	\end{equation*}
	Hence we get that for every $t>0$ there is
	\begin{equation*}
		0 = \lim_{n\to \infty} \langle F(u_n,v_n), w\rangle \geq \lim_{n\to \infty} \langle F(\nu_t,v_n), w\rangle = \langle F(\nu_t, v), w\rangle.
	\end{equation*}
	Letting $t\to 0$ we obtain
	\begin{equation*}
		0 \geq \langle F(u,v), w\rangle\quad \text{for every }w\in X.
	\end{equation*}
	Hence $F(u,v) = 0$. Now, since $u_n \rightharpoonup u$ and $v_n \to v$ we also get ${G(u_n,v_n) \to G(u,v)}$. Therefore $(u,v)$ solves \eqref{MainSystem}.
\end{proof}

%To get rid of the type of continuity given by \eqref{ContinuityCondition} and imposed in Assumption \ref{AssumptionH3} we require that the solution mapping $S$ should be continuous. This assertion is in turn obtained when operator $F(\cdot,v)$ for all $v\in Y$ satisfies condition (S) or related. Hence, again after \cite{BeldzGalKos}, we recall a following result.
\begin{proposition}
    \label{PropositionPreviousPaper2}
    Assume that $Y$ is a metric space, $X$ is a reflexive Banach space. If $F \colon X \times Y\longrightarrow X^*$ is an operator such that:
    \begin{itemize}
        \item $F(\cdot,y)$ is radially continuous for all $y\in Y$;
        \item $F(u,\cdot)$ is continuous for every $u\in X$;
	\item there exists a constant $m>0$ such that 
 \begin{equation}
    \label{ConditionParametricStronglyMonotone}
     \langle F(u,y) - F(w,y), u - w \rangle \geq m\|u - w\|^2 \quad \text{for all }y\in Y\text{ and }u,w\in X,
 \end{equation}
	\end{itemize}
	then for every $y\in Y$ there exists a unique $u_y$ such that $F(u_y,y) = 0$. Moreover $y_n\to y$ in $Y$ implies $u_{y_n} \to u_y$ in $X$.
\end{proposition}
Consequently we obtain a following version of Theorem \ref{TheoremNonLinearEigenvalue}.
\begin{proposition}
    \label{PropositionNonLinearEigenvalue2}
    Let Assumption \ref{AssumptionH1} holds. Assume that:
    \begin{itemize}
        \item $F\colon X\times Y \longrightarrow X^*$ is an operator such that:
        \begin{itemize}
            \item $F(\cdot,v)$ is radially continuous for every $v\in Y$,
            \item $F(u,\cdot)$ is continuous for every $u\in X$,
            \item for every $r > 0$ we have $\sup_{\|y\|\leq r}\|F(0,y)\| < \infty$,
            \item there exists $m>0$ such that \eqref{ConditionParametricStronglyMonotone} holds,
        \end{itemize}
        \item $G\colon X \times Y \longrightarrow Y$ is continuous and compact.
    \end{itemize}
    If there exists a bounded and convex set $C\subset Y$ such that $G(X\times C) \subset C$, then there exists at least one solution to \eqref{MainSystem}.
\end{proposition}
\begin{proof}[Sketch of the proof]
    While we argue analogously as in the proof of Theorem \ref{TheoremNonLinearEigenvalue} we do not need to perturb $F$ by $\frac{1}{n}J$. This is because of the monotonicity properties of operator $F(\cdot,v)$ for all $v\in Y$.  Applying Proposition \ref{PropositionPreviousPaper2} directly to $F$ we obtain a continuous mapping $S\colon Y \longrightarrow X$ satisfying $F(S(v),v) = 0$ for all $v\in Y$. Condition \eqref{ConditionParametricStronglyMonotone} yields 
    \begin{equation*}
        0 = \langle F(S(v),v),S(v)\rangle \geq m\|S(v)\|^2 - \|F(0,v)\|\|S(v)\|.
    \end{equation*}
    Consequently $\|F(0,v)\| \geq m\|S(v)\|$ for all $v\in K$ and hence set $K:= \{S(v) : v\in C\}$ is bounded. Applying the Schauder Fixed Point Theorem to $\widetilde{G}\colon C \longrightarrow C$ given by $\widetilde{G}(v) = G(S(v),v)$ we get the assertion.
\end{proof}

\begin{remark}
    \label{RemarkAfterFirstTheorem}
    In the above arguments, the Schauder Theorem can be replaced by the Scheafer Theorem. This would require some relevant change of the assumptions but the main spirit will be retained, see \cite{PrecupSchaefer2}.
\end{remark}

\begin{remark}
	To consider problem 
	\begin{equation*}
		\left\{
		\begin{array}{l}
			F(u,v) = f, \\
			G(u,v) = v 
		\end{array}
		\right.
	\end{equation*}
	for a fixed  $f\in X^*$ it is sufficient to replace $F$ by $\widetilde{F}(u,v) = F(u,v) - f$ and take ${\widetilde{\gamma}(x,y) = \gamma(x,y) - \|f\|_*x}$ since
	\begin{equation*}
		\langle \widetilde{F}(u,v),u\rangle = \langle F(u,v),u\rangle - \langle f,u\rangle \geq \gamma(\|u\|,\|v\|) - \|f\|_*\|u\| = \widetilde{\gamma}(\|u\|,\|v\|).
	\end{equation*}
\end{remark}

Notice that assumption $G(X\times C)\subset C$ is very restrictive. For instance, if $G\colon C[0,1]\times C[0,1]\longrightarrow C[0,1]$ is an integral operator associated with equation
\begin{equation*}
	\left\{
	\begin{array}{ll}
		-\ddot{v} = g(t,u,v) & \text{for }t\in (0,1),\\
		v(0) = v(1) = 0, &
	\end{array}
	\right.
\end{equation*}
then a natural assumption providing $G(X\times C)\subset C$ for some bounded set $C$ is a uniform boundedness of $g$ with respect to the first coordinate. An inspiration for a more applicable condition will be found in the Krasnoselskii Theorem.
%%Here $k\colon \left[ 0,1\right] \times \left[ 0,1\right] \longrightarrow \mathbb{R}$ is defined by
%%\begin{equation*}
   % k\left( t,s\right) = 
    %\left\{\begin{array}{cc}
    %    \left( 1-t\right) s, & 0\leq s\leq t, \\ 
    %    t\left( 1-s\right) , & t<s\leq 1.
    %\end{array}
  %  \right.
%\end{equation*}
%and $G$ is given by 
%%\begin{equation*}
%    G\left( u,v\right))  =\int_{0}^{1}k\left( t,s\right) g\left( s,u\left( s\right) ,v\left( s\right) \right) ds.
%\end{equation*}

\subsection{Relations to the Krasnoselskii fixed point theorem}

It is shown in \cite{PrecupFixed} that Theorem \ref{TheoOFAvramescu} implies the following fixed point of Krasnoselskii which glues together the Banach and the Schauder Fixed Point Theorems:

\begin{theorem}
	Let $D$ be a closed bounded convex subset of a Banach space $X$, $A:D\longrightarrow X$ a contraction and $B:D\longrightarrow X$ a continuous mapping with $B(D)$ relatively compact. If \begin{equation*}
		A(x)+B(y)\in D\quad\text{for all }x,y\in D
	\end{equation*}
	then the mapping $A+B$ has at least one fixed point.
\end{theorem}

There are extensions of the Krasnosel'skii theorem in several directions, among which we may mention those pertaining to replacing $B$ with some injection or else to the usage of the Schaefer fixed point theorem. Both would lead to suitable modification of the assumptions, see for example \cite{avramescuEJQTDE}, \cite{AvramVlad} for the relevant results.
\newline
\par
Our version of the Krasnoselskii fixed point theorem, provided below, is connected to the usage of the Strongly Monotone Principle instead of the Banach contraction and uses ideas employed in \cite{burton}. We say that $A \colon H \longrightarrow H$, where $H$ is a real Hilbert space, is \emph{one-sided contraction} if there exists $m < 1$ such that 
\begin{equation}
	\label{OneSidedContraction}
	\langle A(u) - A(w), u - w\rangle \leq m\|u - w\|^2\quad \text{for all }u,w\in H.
\end{equation} 
Now some version of the fixed point theorem dependent on a numerical parameter follows.

\begin{theorem}
    \label{theo_on_a_ball}
    Assume that $A\colon H \longrightarrow H$ is a radially continuous one-sided contraction and ${B\colon H \longrightarrow H}$ is continuous and compact. Then there exists $\lambda_0>0$ such that for all $\lambda \in [0,\lambda_0]$ the mapping ${u\longmapsto A\left( u\right) -\lambda B\left( u\right)}$ has a fixed point, or in other words, equation 
	\begin{equation}
		\label{ProblemEigenvalue}
		A\left( u\right) =\lambda B\left( u\right) +u
	\end{equation}
	has a solution.
\end{theorem}

\begin{proof}
	Take $r = \frac{\|A(0)\|}{1 - m}$ and define $P_r\colon H \longrightarrow H$ by 
	\begin{equation*}
		P_r(u) := \left\{
		\begin{array}{cl}
			u & \text{if }\|u\|\leq r, \\
			\frac{u}{\|u\|} & \text{if }\|u\| > r.
		\end{array}
		\right.
	\end{equation*}
	Put 
	\begin{equation*}
	    \lambda_0 := \frac{1}{1+\sup_{\|u\|\leq r}\|B(u)\|}
	\end{equation*}
	and fix $\lambda < \lambda_0$. Let $X = Y = H$, $F(u,v) = u - A(u) - v$, and $G(u,v) = \lambda B(P_r(u))$. Then for all $u,w,v\in X$ we get 
    \begin{equation*}
        \langle F(u,v) - F(w,v), u - w\rangle = \|u - w\|^2 - \langle A(u) - A(w), u - w\rangle \geq (1 - m)\|u - v\|^2.
    \end{equation*}
    Moreover $F(0,v) = v$ and hence $F$ satsisfies the assumptions of Proposition \ref{PropositionNonLinearEigenvalue2}. Operator $G$ is clearly continuous and compact since $B$ is a compact mapping. Identifying $H$ with $H^*$ via the Riesz Representation, we can apply Proposition \ref{PropositionNonLinearEigenvalue2} and obtain that there exists $u_0\in H$ such that
	\begin{equation*}
		u_0 - A(u_0) = \lambda B(P_r(u_0)).
	\end{equation*}
	Therefore using \eqref{OneSidedContraction} we get
	\begin{equation*}
		\|u_0\|((1 - m)\|u_0\| - \|A(0)\|) \leq \langle u_0 - A(u_0), u_0\rangle = \lambda \langle B(P_r(u_0)),u_0\rangle \leq \|u_0\|,
	\end{equation*}
	which gives $\|u_0\| \leq r$ and hence $u_0$ solves \eqref{ProblemEigenvalue}. Since $\lambda$ was taken arbitrary from $[0,\lambda_0]$, we get the assertion.
\end{proof}

% \begin{remark}
	%     We might have assumed that $A\colon H\longrightarrow H$ is an $m-$strongly monotone with $m>1$ and next argue that $(A-I)^{-1}\colon H\longrightarrow H$ is continuous which is rather artificial in the concerete applications. The critical case of $m=1 $ cannot be covered by this approach due to the fact that the identity is not compact.
	% \end{remark}

\begin{remark}
	Following Remark \ref{RemarkAfterFirstTheorem} we can obtain the result which utilizes the Schaefer fixed point theorem: Let $H$ be a Hilbert space, $A\colon H\longrightarrow H$ be a radially continuous one-sided contraction and $B\colon H\longrightarrow H$ be a continuous compact operator. Then either
	\begin{enumerate}
		\item $x=\lambda A(x/\lambda )+\lambda Bx$ has a solution in $H$ for $\lambda =1$
		\item the set of all such solutions, $0<\lambda <1$, is unbounded.
	\end{enumerate}
	The proof is exactly as above and relies on the observation that for any $\lambda \in (0,1)$ mapping $x\longmapsto \lambda A(x/\lambda )$ defines one-sided contraction (independent of such $\lambda $ and same as for mapping $x\longmapsto A(x)$).
\end{remark}

As is mentioned, for example in \cite{ButronKirk}, the main problem about checking the assumptions of the Krasnoselskii Theorem is the invariance condition. In Theorem \ref{theo_on_a_ball} we imposed instead some condition on the numercial parameter which however in direct applications may become rather small and depending on the behaviour of the mapping $B$ on a ball. Therefore we propose some approach, suggested also in \cite{avramescuEJQTDE}, how to deal with this issue.

\begin{theorem}
	Let $D$ be a closed convex subset of a Hilbert space $H$, $A\colon H\longrightarrow H$ be a radially continuous one-sided contraction and $B\colon D\longrightarrow H$ be continuous with $B(D)$ relatively compact. If
	\begin{equation}
		u=A(u)+B(v)\text{ and }v\in D\text{ imply that }u\in D  \label{condKras}
	\end{equation}
	then the mapping $A+B$ has at least one fixed point.
\end{theorem}

\begin{proof}
	We see, by the Strongly Monotone Principle, see \cite{galewskiBOOKS}, that mapping ${I - A\colon H\longrightarrow H}$ is a bijection with a continuous inverse. This means that for any $v\in D$ there is exactly one $u\in H$ such that $u=A(u)+B(v)$. Due to \eqref{condKras} we see that $u\in D$. Hence for the operator 
	\begin{equation}
		\label{DefinitionOfT}
		T(u) := (I-A)^{-1}\left(B(u)\right) \text{.}
	\end{equation}
	we see that $T\colon D\longrightarrow D$. Since $T\left( D\right) $ is relatively compact as well, we get the assertion by the Schauder Fixed Point Theorem.
\end{proof}

Inspired by condition \eqref{condKras} we extend Theorem \ref{TheoremNonLinearEigenvalue} as follows.

\begin{theorem}
	\label{TheoremExtended}
	Let Assumptions \ref{AssumptionH1}, \ref{AssumptionH2} and \ref{AssumptionH3} hold. Moreover assume that there exists a function ${\psi\colon [0,\infty)^2\longrightarrow [0,\infty)}$ such that
	\begin{equation*}
		\|G(u,v)\|\leq \psi(\|u\|,\|v\|)\quad \text{for all }u \in X \text{ and }v\in Y;
	\end{equation*} 
	If there exists $R>0$ such that for any $x \in \mathbb{R}$
	\begin{equation}
		\label{StrangeAssumption}
		\gamma(x,y)\leq 0\text{ and }y \leq R\text{ imply that } \psi(x,y) \leq R.
	\end{equation}
	then system \eqref{MainSystem} has at least one solution.
\end{theorem}

The following remark explains connections between conditions \eqref{condKras} and \eqref{StrangeAssumption}.
\begin{remark}
	Condition \eqref{condKras} reads as follows: 
	\begin{center}
		\textit{Whenever $v\in D$ and $u$ is a solution to $u = A(u) + B(v)$, then $u\in D$.}
	\end{center}
	It can be reformulated in terms of operator $T$ defined by \eqref{DefinitionOfT} as follows
	\begin{center}
		\textit{Whenever $v\in D$ and $u$ is a solution to $u = A(u) + B(v)$, then $T(v)\in D$.}
	\end{center}
	Now it is easy to see an analogy between condition \eqref{condKras} and the following one.
	\begin{center}
		\textit{Whenever $v\in B_R$ and $u$ is a solution to $F(u,v) = 0$, then $G(u,v)\in B_R$.}
	\end{center}
	Here $B_R := \{v\in Y : \|v\|\leq R\}$. To express this condition in terms of $\gamma$ and $\psi$ (described in Assumption \ref{AssumptionH2} and Theorem \ref{TheoremExtended}, respectively) let us observe that if $u$ is a solution to $F(u,v) = 0$, then
	\begin{equation*}
		\gamma(\|u\|,\|v\|) \leq \langle F(u,v),u\rangle = 0.
	\end{equation*}
	Moreover, condition $\psi(\|u\|,\|v\|) \leq R$, provides $\|G(u,v)\| \leq R$, that is $G(u,v) \in B_R$. Hence we see that \eqref{StrangeAssumption} is (at least in some sense) a counterpart of assumption \eqref{condKras}.
\end{remark}

\begin{proof}[Proof of Theorem \ref{TheoremExtended}]
	Let us denote $B_R := \{v\in Y : \|v\|\leq R\}$. Define $F_n$, \eqref{FnProof} and $S_n$ as in the proof of Theorem \ref{TheoremNonLinearEigenvalue}. Let $G_n \colon B_R \longrightarrow Y$ be given by
	\begin{equation*}
		G_n(v) = G(S_n(v),v).
	\end{equation*}
	$G_n$ is clearly continuous. Moreover, since $\|v\|\leq R$ and since $\gamma$ is uniformly coercive on each bounded interval, set $K := \{S_n(v) : n\in \mathbb{N},\ v\in B_R\}$ is bounded. Therefore 
	\begin{equation*}
		G_n\left(B_R\right) \subset G\left(K\times B_R\right)
	\end{equation*}
	and hence $G_n$ is compact. Now we show that $G_n\colon B_R \longrightarrow B_R$. Let $\|v\|\leq R$. Then
	\begin{equation*}
		\|G_n(v)\| = \|G(S_n(v),v)\|\leq \psi(\|S_n(v)\|,\|v\|) 
	\end{equation*}
	Moreover
	\begin{equation*}
		0 = \langle F_n(S_n(v),v), S_n(v)\rangle = \langle F(S_n(v),v), S_n(v)\rangle + \tfrac{1}{n}\langle J(S_n(v)),S_n(v)\rangle  \geq \gamma(\|S_n(v)\|,\|v\|)
	\end{equation*}
	and hence $\|G_n(v)\|\leq R$ by \eqref{StrangeAssumption}. By the Schauder Fixed Point Theorem, applied to $G_n$, there exists $v_n\in B_R$ such that $G_n(v_n) = v_n$. Take $u_n := S_n(v_n)$. Then $(u_n,v_n)$ solves \eqref{AuxiliarySystem}. Since both, $(u_n)$ and $(v_n)$, are bounded, arguments following relation (4) in the proof of Theorem \ref{TheoremNonLinearEigenvalue} provide the assertion.
\end{proof}

There are also other results pertaining to the Krasonselskii Theorem. The Authors in \cite{Krysz} introduce the class of single- and set-valued Krasnosel'skij-type maps for which they construct a fixed point index theory next applied to constrained differential inclusions and equations.

\section{Applications to the nonlinear systems}

We give applications for a system of mixed nonlinear and semilinear Dirichlet problems for ODE. Firstly, we will describe an abstract framework for perturbed $p$-Laplacian and nonlocal $q$-Laplacian equation. Then we will apply Theorem \ref{TheoremExtended} to obtain solvability of system of equations.

\subsection{Perturbed $p$-Laplacians}

We need some preparation about the nonlinear perturbed Laplacian which we consider and which pertains to the perturbed $p$-Laplacian. It seems that the presented results can be extended to study problems involving $p(\cdot)$-Laplacian with some necessary technical modifications.
\newline
\par
For $p\geq 2$ and $\varphi \colon \left[ 0,1\right] \times \mathbb{R}\longrightarrow \mathbb{R}$ we define  ${D\colon W_{0}^{1,p}(0,1)\times C[0,1]\longrightarrow W^{-1,p^{\prime }}(0,1)}$ by 
\begin{equation*}
	\langle D(u,v),w\rangle =\int_{0}^{1}\varphi \big(t,v(t),|\dot{u}(t)|^{p-1}\big)|\dot{u}(t)|^{p-2}\dot{u}(t)\dot{w}(t)dt.
\end{equation*}

Consider the following hypotheses on $\varphi $, which will lead to the well posedness, continuity, coercivity and monotonicity properties of operator $D$.

\begin{assumption}
	\label{AssumptionPhi}
	Let $\varphi\colon[0,1] \times \mathbb{R}\times [0,\infty) \longrightarrow \mathbb{R}$ and assume that there exist continuous functions $m,M \colon [0,\infty)\longrightarrow (0,\infty)$, {$m$ -- nonincreasing}, such that:
	\begin{enumerate}
		\item[($\Phi$1)]\ $\varphi(\cdot,y,r)$ is Lebesgue measurable for all $y\in \mathbb{R}$ and every $r\geq 0$;
		\item[($\Phi$2)]\ $\varphi(t,\cdot,r)$ and $\varphi(x,y,\cdot)$ are continuous for a.e. $t\in (0,1)$, all $y\in \mathbb{R}$ and every $r\geq 0$;
		\item[($\Phi$3)]\ $m(|y|) \leq \varphi(t,y,r) \leq M(|y|)$ for a.e. $t\in (0,1)$, all $y\in \mathbb{R}$ and every $r\geq 0$;
		\item[($\Phi$4)]\ $\varphi(t,y,r)r \leq \varphi(t,y,s)s$ for a.e. $t\in (0,1)$, all $y\in \mathbb{R}$ and every $s \geq r \geq 0$. 
	\end{enumerate}
\end{assumption}

For a given function $f\colon [0,1]\times \mathbb{R}^2\longrightarrow \mathbb{R}$ we define \emph{the Nemyskii operator} ${N_f\colon W^{1,p}_0(0,1)\times C[0,1]\longrightarrow W^{-1,p'}(0,1)}$ given by the formula 
\begin{equation*}
	\langle N_f(u,v), w \rangle = \int_0^1 f(t, u(t), v(t)) w(t)dt.
\end{equation*}
To obtain well posedness and continuity of $N$ we consider
\begin{assumption}
	\label{AssumptionF}
	A function $f\colon [0,1]\colon \mathbb{R}\times\mathbb{R}\longrightarrow\mathbb{R}$ is such that:
	\begin{itemize}
		\item[(F1)]\ $f(\cdot, u,v)$ is Lebesgue measurable for all $u,v\in \mathbb{R}$;
		\item[(F2)]\ $f(t,\cdot,v)$ is continuous and nonincreasing for a.e. $t\in [0,1]$ and all $v\in \mathbb{R}$;
		\item[(F3)]\ $f(t,u,\cdot)$ is continuous for a.e. $t\in [0,1]$ and all $u\in \mathbb{R}$.
	\end{itemize}
	Moreover
	\begin{itemize}
		\item[(F4)]\ there exists a function $\delta \colon [0,\infty)\longrightarrow [0,\infty)$ satisfying
		\begin{equation*}
			\sup_{\substack{0\leq t\leq 1\\ -v\leq y\leq v}}|f(t,0,y)|\leq \delta(v)\quad \text{for every }v\geq 0.
		\end{equation*}
	\end{itemize}
\end{assumption}

\begin{lemma}
	\label{LemmaOperatorDpPhi} 
	If function $\varphi$ satisfies Assumption \ref{AssumptionPhi} and function $f$ satisfy Assumption \ref{AssumptionF}, then operator ${F = D - N_f}$ satisfies Assumption \ref{AssumptionH2} with $\gamma(x,y) = m(y)x^p - \frac{1}{\lambda_p}\delta(y)x$.
\end{lemma}
\begin{proof}
    ($\Phi$1), ($\Phi$2), ($\Phi$3), (F1), (F2) and (F3) provide well posedness of $D$ and $N_f$. Monotonicity of $D$ and $-N_f$ follows by ($\Phi$4) and (F2), respectively. See \cite{BeldzGalKos} or \cite{galewskiBOOKS} for details. Finally we show by a direct calculation that $D-N_f$ is coercive by ($\Phi$3) and (F4). Notice that for every $u\in W^{1,p}_0(0,1)$ and all $v\in C[0,1]$ we have
	\begin{equation*}
		\begin{split}
			\langle D(u,v), u\rangle & = \int_0^1 \varphi (t,v(t),|\dot{u}|^{p - 1})|\dot{u}(t)|^p dt \geq \int_0^1 m(|v(t)|)|\dot{u}(t)|^p dt\\
			&\geq m(\|v\|_\infty)\int_0^1 |\dot{u}(t)|^p dt = m(\|v\|_\infty)\|u\|_p^p.
		\end{split}
	\end{equation*}
	and
	\begin{equation*}
		\begin{split}
			-\langle N_f(u,v),u\rangle & = -\int_0^1 f(t,u(t),v(t))u(t) dt \geq -\int_0^1 f(t, 0, v(t)) u(t) dt\\ 
			& \geq -\left(\int_0^1 |f(t,0,v(t))|^{p'}dt\right)^\frac{1}{p'}\left(\int_0^1 |u(t)|^{p}dt\right)^\frac{1}{p} \\
			& \geq -\tfrac{1}{\lambda_p}\left(\int_0^1|\delta(\|v\|_\infty)|^{p'}dt\right)^\frac{1}{p'}\|u\|_p \geq -\tfrac{1}{\lambda_p} \delta(\|v\|_\infty)\|u\|_p.\qedhere
		\end{split}
	\end{equation*}
\end{proof}
Notice that $D(u,v) - N_f(u,v) = 0$ if and only if $u$ is a weak solution to the following nonlinear boundary value problem
\begin{equation*}
	\left\{
	\begin{array}{ll}
		\frac{d}{dt}\left(\varphi\big(t,v(t),|\dot{u}(t)|^{p-1}\big)|\dot{u}(t)|^{p-2}\dot{u}(t)\right) = f(t, u(t), v(t)) & \text{for }t\in (0,1),\\
		u(0) = v(0) = 0,
	\end{array}
	\right.
\end{equation*}
that is if $u\in W^{1,p}_0(0,1)$ satisfies 
\begin{equation*}
    \int_0^1 \varphi\big(t,v(t),|\dot{u}(t)|^{p-1}\big)|\dot{u}(t)|^{p-2}\dot{w}(t) dt = \int_0^1  f(t, u(t), v(t))w(t) dt\quad \text{for every }w\in W^{1,p}_0(0,1).
\end{equation*}
Using \emph{du Bois-Reymond's Lemma}, see \cite{MawhinProblemesDeDirichlet}, it can be proved that $u$ is a weak solution defined above, function $|\dot{u}(\cdot)|^{p-2}\dot{u}(\cdot)$ is weakly differentiable and hence $f(\cdot,u(\cdot),v(\cdot))$ is a weak derivative of $|\dot{u}(\cdot)|^{p-2}\dot{u}(\cdot)$. Moreover function $|\dot{u}(\cdot)|^{p-2}\dot{u}(\cdot)$ is differentiable almost everywhere in a classical sense. Hence
\begin{equation*}
    -\frac{d}{dt}\left(|\dot{u}(t)|^{p-2}\dot{u}(t)\right) = f(t,u(t),v(t))\quad \text{for a.e. }t\in [0,1].
\end{equation*}

\subsection{$q$-Laplace equation with non-local boundary conditions}
For fixed $q > 1$ and $u\in C[0,1]$ we consider the following nonlinear system 
\begin{equation}\label{NonSys}
	\left\{
	\begin{array}{ll}
		\displaystyle -\frac{d}{dt}\big(|\dot{v}(t)|^{q-2}\dot{v}(t)\big)=g(t,u(t),v(t)) & \text{for }t\in (0,1), \\ 
		\\
		\displaystyle v(0)=\int_0^1h_0(v(s))dA_0(s), \quad v(1)=\int_0^1 h_1(v(s))dA_1(s). &
	\end{array}
	\right.
\end{equation}
Solutions to \eqref{NonSys} are understood in the classical sense, namely a differentiable function $v\colon [0,1] \longrightarrow \mathbb{R}$ is a solution to \eqref{NonSys} if $|\dot{v}(\cdot)|^{q-2}\dot{v}(\cdot)$ is differentiable and if \eqref{NonSys} holds. To study \eqref{NonSys} using the Fixed Point Theory we impose
\begin{assumption}
	\label{AssumptionG}\ 
	\begin{itemize}
	    \item[(G0)]\ Let $g\colon [0,1]\times \mathbb{R}^2 \longrightarrow \mathbb{R}$ and $h_0,h_1\colon \mathbb{R}\longrightarrow \mathbb{R}$ be continuous functions and let $A_0,A_1\colon[0,1]\longrightarrow\mathbb{R}$ have a bounded variation.
		\item[(G1)]\ There are numbers $A,B,C,r\geq 0$ and $0\leq \theta <q-1$ such that
		\begin{equation*}
			g(x,u,v)\leq A|u|^r+B|v|^{\theta} +C \quad \text{for all }x \in [0,1]\text{ and sufficiently large }u,v\in \mathbb{R}.
		\end{equation*}
		\item[(G2)]\ There exist numbers $\alpha_j,\beta_j\geq0$ such that $|h_j(v)|\leq \alpha_j|v| + \beta_j$ for all $v\in\mathbb{R}$, $j=0,1$.
		\item[(G3)]\ Either $\big(2\alpha_0 \operatorname{Var}{A_0}+ \alpha_1 \operatorname{Var}{A_1}\big)<1$ or $\big(2\alpha_1 \operatorname{Var}{A_1}+ \alpha_0 \operatorname{Var}{A_0}\big)<1$, where $\operatorname{Var}{A_j}$ stands for the variation of the function $A_j$, $j=0,1$.
	\end{itemize}
\end{assumption}
Integrals in \eqref{NonSys} are understood in the Riemann-Stieltjes sense. By $V$ we denote the Volterra integral operator, namely
\begin{equation}
	Vv(t)=\int_0^tv(s) ds\quad \text{for every }t\in [0,1].
\end{equation}
Moreover, we define the Nemytskii operator associated with $g$ analogously as before but in a different function setting. Let $N_g\colon (C[0,1])^2 \longrightarrow C[0,1]$ be given by the formula
\begin{equation*}
    N_g(u,v)(t) = g\big(t, u(t),v(t)\big).
\end{equation*}

The function associated with the $q$-Laplacian is denoted by $\psi_q$, that is
\begin{equation*}
	\psi_{q}(\zeta)=|\zeta|^{q-2}\zeta \quad \text{for every }\zeta \in \mathbb{R}.
\end{equation*}
To properly define integral operator associated with \eqref{NonSys} we need
\begin{lemma}\label{Func} We assume condition (G0) from Assumption \ref{AssumptionG}.
	For every $u,v\in C[0,1]$ there is exactly one $c=c(u,v)$ such that 
	\begin{equation*}
		\int_0^1\psi_q^{-1}\big(c(u,v)-VN_g(u,v)(s)\big)ds = \int_0^1 h_1(v(s))dA_1(s)-\int_0^1 h_0(v(s))dA_0(s).
	\end{equation*}
	Moreover, the mapping $C[0,1]\times C[0,1]\ni (u,v)\longmapsto c(u,v)\in \mathbb{R}$ is continuous.
\end{lemma}
\begin{proof}
	For fixed $u,v\in C[0,1]$ we define 
	\begin{equation*}
		\Theta_{(u,v)}(c)=\int_0^1\psi_q^{-1}\big(c(u,v)-VN_g(u,v)\big)ds + \int_0^1 h_0(v(s))dA_0(s) - \int_0^1 h_1(v(s))dA_1(s).
	\end{equation*} 
	Since $\psi_q$ is continuous and strictly increasing, then so is function $\Theta_{(u,v)}$. Moreover, we have ${\lim_{|c|\to\infty}|\Theta_{(u,v)}(c)|=\infty}$. Hence $\Theta_{(u,v)}(c)=0$ for a unique $c=c(u,v)$. Next, the monotonicity of $\psi_q$ yields
	\begin{equation*}
		\psi_q^{-1}\big(c(u,v)-\|N_g(u,v)\|_{\infty}\big)\leq\int_0^1 \psi_q^{-1}\big(c(u,v)-VN_g(u,v)(s)\big)ds \leq \psi_q^{-1}\big(c(u,v)+\|N_g(u,v)\|_{\infty}\big).
	\end{equation*}
	Hence 
	\begin{equation*}
		\begin{split}
			\psi_q \left(\int_0^1 h_1(v(s))dA_1(s)-\int_0^1 h_0(v(s))dA_0(s)\right) - \|N_g(u,v)\|_{\infty} & \leq c(u,v) \\
			\leq \psi_q\left(\int_0^1 h_1(v(s))dA_1(s)-\int_0^1 h_0(v(s))dA_0(s)\right) & + \|N_g(u,v)\|_{\infty}.
		\end{split}
	\end{equation*}
    Now, let $u_n\to u_0$ and $v_n \to v_0$ in $C[0,1]$ and suppose that $c(u_n,v_n) \not\to c(u_0,v_0)$. Then, since $(c(u_n,v_n))$ is bounded, we see that $c(u_n,v_n)\to c_*\neq c(u_0,v_0)$, up to the subsequence. Moreover, $VN_g(u_n,v_n)\to VN_g(u_0,v_0)$ in $C[0,1]$. Hence we obtain $\psi_q^{-1}\big(c(u_n,v_n)-VN_g(u_n,v_n)(\cdot)\big)\to \psi_q^{-1}\big(c_*-VN_g(u_0,v_0)(\cdot)\big)$ in $L^1(0,1)$. Since $h_0(v_n(\cdot))\to h_0(v(\cdot))$, $h_1(v_n(\cdot))\to h_1(v(\cdot))$ uniformly on $[0,1]$, we get 
	\begin{equation*}
		\begin{split}
			0 & = \lim_{n\to\infty}\left(\int_0^1 \psi_q^{-1}\big(c(u_n,v_n)-VN_g(u_n,v_n)(s)\big)ds+\int_0^1 h_0(v_n(s))\,dA_0(s) -\int_0^1 h_1(v_n(s))dA_1(s)\right) \\
			& =\int_0^1 \psi_q^{-1}\big(c_*-VN_g(u_0,v_0)(s)\big)ds+\int_0^1 h_0(v_0(s))dA_0(s) -\int_0^1 h_1(v_0(s))dA_1(s).
		\end{split}
	\end{equation*}
    By uniqueness of $c(u_0,v_0)$ we need to have $c_*=c(u_0,v_0)$, which is a desired contradiction. Hence ${c(u_n,v_n) \to c(u_0,v_0)}$ and the continuity of $c$ is proved.
\end{proof}
Now we define operator $T\colon W^{1,p}_0(0,1)\times C[0,1]\longrightarrow C[0,1]$ by the formula
\begin{equation}
	\label{OperatorTDifferential}
	T(u,v)(t)=\int_0^t\psi_q^{-1}\big(c(u,v)-VN_g(u,v)(s)\big)ds+\int_0^1h_0(v(s))dA_0(s),
\end{equation}
where $c$ is defined in Lemma \ref{Func}.
\begin{lemma}\label{Oper} We assume condition (G0) from Assumption \ref{AssumptionG}.
	For every $u\in W^{1,p}(0,1)$, the function $v\in C[0,1]$ is a solution to \eqref{NonSys} if and only if it is fixed point of operator $T(u,\cdot)$.
\end{lemma}
\begin{proof}[Sketch of proof]
    For fixed $u\in W^{1,p}(0,1)$, if $v\in C[0,1]$ is a solution to \eqref{NonSys} we have
    \begin{equation*}
        v(t) = \int_0^1 h_0(v(s))dA_0(s) + \int_0^t \psi_q^{-1}\big(\psi(\dot{v}(0) -  VN_g(u,v)(s)\big).
    \end{equation*}
    By boundary condition for $v(1)$ we obtain
    \begin{equation*}
        \int_0^1\psi_q^{-1}\big(c(u,v)-VN_g(u,v)(s)\big)ds = \int_0^1 h_1(v(s))dA_1(s)-\int_0^1 h_0(v(s))dA_0(s),
    \end{equation*}
    which means that $\dot{v}(0)=c(u,v)$. This gives $v=T(u,v)$. On the other hand if $u\in W^{1,p}(0,1)$, $v\in C[0,1]$ and $v=T(u,v)$ then $v$ is $C^1$ and
    \begin{equation*}
        \dot{v}(t) = \psi^{-1}_q\big (c(u,v) - VN_g(u,v)(t)\big).
    \end{equation*}
    Therefore $|\dot{v}(\cdot)|^{q-2}\dot{v}(\cdot)$ is differentiable and
    \begin{equation*}
        -\frac{d}{dt}\big(|\dot{v}(t)|^{q-2}\dot{v}(t)\big)=g(t,u(t),v(t)).
    \end{equation*}
    Moreover, it is easy to observe that $v$ satisfies the given boundary conditions.
\end{proof}

Via standard argumentations we get the following Lemma.
\begin{lemma}\label{Com} We assume condition (G0) from Assumption \ref{AssumptionG}.
	For fixed $u\in W^{1,p}_0(0,1)$, the operator $T(u,\cdot)\colon C[0,1]\longrightarrow C[0,1]$ is continuous and compact.
\end{lemma}
\begin{proof}
Let us fix $u\in W^{1,p}_0(0,1)$. Lemma \ref{Func} yields the continuity of $c(u,\cdot)$. As $VN_g(u,\cdot)$ is a composition of the Nemytskii operator and the Volterra integral operator, we see it is continuous map. Next, using the continuity of $\psi^{-1}_q$ and $h_0$, we get that $T(u,\cdot)$ is also continuous. \par 
To obtain the compactness of $T(u,\cdot)$, firstly we observe that $VN_g(u,\cdot)$ is compact since it is a composition of a continuous and bounded operator with the compact one. Since $c(u,\cdot)$ has the values in $\mathbb{R}$, the mapping 
$c(u,\cdot)-VN_g(u,\cdot)$ is compact. Superposition of above map with the Nemytskii operator associated with continuous function $\psi^{-1}_q$ and next with the Volterra integral operator is also compact. Therefore, the first term of $T(u,\cdot)$ is compact. Finally, since range of $v\mapsto \int\limits_{0}^{1}h_0(v(s))ds$ lies in real line, the operator $T(u,\cdot)$ is compact.
\end{proof}
\begin{theorem}
	Let Assumption \ref{AssumptionG} hold. Then for every $u\in W^{1,p}_0(0,1)$ the problem \eqref{NonSys} admits at least one solution.
\end{theorem}
\begin{proof}

We assume for the proof the first condition in (G3) holds. The remaining case follows likewise. 
	Since assumption (G2) holds we get 
	\begin{equation}
		\label{h0}
		\bigg|\int_0^1h_0(v(s))dA_0(s)\bigg|\leq (\alpha_0\|v\|_{\infty} + \beta_0)\operatorname{Var}{A_0}.
	\end{equation}
	According to the proof of Lemma \ref{Func} and assumptions (G1), (G2), we have 
	\begin{equation*}
	    \begin{split}
		    |c(u,v)-N_g(u,v)(s)|\leq & \bigg|\psi_q\bigg(\int_0^1 h_1(v(s))dA_1(s)-\int_0^1 h_0(v(s))dA_0(s)\bigg)\bigg|+2\|N_g(u,v)\|_{\infty} \\
		    \leq &\psi_q\Big(\big(\alpha_1\operatorname{Var}{A_1}+\alpha_0\operatorname{Var}{A_0}\big)\|v\|_{\infty} + \beta_1\operatorname{Var}{A_1}+\beta_0\operatorname{Var}A_0\Big)\\
		    & +2\left(A\|u\|^{r}_{\infty} + B\|v\|^{\theta}_{\infty}+C\right).
	    \end{split}
	\end{equation*}
	Above estimation combined with \eqref{h0} give a following inequality 
	\begin{equation*}
	    \begin{split}
		    \|T(u,v)\|_\infty\leq& \int_0^1|\psi_q^{-1}\big(c(u,v)-VN_g(u,v)(s)\big)|\,ds+\bigg|\int_0^1h_0(v(s))\,dA_0(s)\bigg| \\
		    \leq& \int_0^{1}\psi^{-1}_q\bigg(\psi_q\Big(\big(\alpha_1\operatorname{Var}{A_1}+\alpha_0\operatorname{Var}{A_0}\big)\|v\|_{\infty}\Big) + \beta_1\operatorname{Var}{A_1} +\beta_0\operatorname{Var}{A_0}+\\
		    &+2(A\|u\|_{\infty}^r+B\|v\|^{\theta}_{\infty}+C)\bigg)ds+\alpha_0\|v\|_{\infty}\operatorname{Var}{A_0}  + \beta_0 \operatorname{Var}{A_0}\\
		    \leq& \big(\alpha_1\operatorname{Var}{A_1}+2\alpha_0\operatorname{Var}{A_0}\big)\|v\|_{\infty}+\\
		    &+\beta_1\operatorname{Var}{A_1} +\beta_0\operatorname{Var}{A_0} +2\Big(A\|u\|^r_\infty+2C\Big)^{\frac{1}{q-1}}+(2B)^{\frac{1}{q-1}}\|v\|_{\infty}^{\frac{\theta}{q-1}},
		\end{split}
	\end{equation*}
	therefore
	\begin{equation}
            \begin{split}
		      \label{BoundednessOfT}
		      \|T(u,v)\|_{\infty}\leq & \big(\alpha_1\operatorname{Var}{A_1}+2\alpha_0\operatorname{Var}{A_0}\big)\|v\|_{\infty} +\beta_1\operatorname{Var}{A_1} \\ 
                &+\beta_0\operatorname{Var}{A_0} +\Big(2A\|u\|^r_\infty+2C\Big)^{\frac{1}{q-1}}+(2B)^{\frac{1}{q-1}}\|v\|_{\infty}^{\frac{\theta}{q-1}}.
            \end{split}
	\end{equation}
	Since assumption (G3) holds and $0\leq \theta_2<q_1$, we have $\|T_u(v)\|_{\infty}\leq R$ for $\|v\|_{\infty}\leq R$ with sufficiently large $R>0$. By Lemma \ref{Com} the operator $T_v$ is completely continuous, hence the existence of solution to \eqref{NonSys} is a consequence of the Schauder Fixed Point Theorem applied to the ball centered at 0 with radius $R$.
\end{proof}

Problems that are considered in this section can be viewed in a framework of some recently introduced in \cite{PrecupNonLocal} control approach towards nonlinear equations investigated by fixed point techniques.

\subsection{System of equations}

To show possible application of Theorem \ref{TheoremExtended} we study solvability of the following system of differential equations.
\begin{equation}
	\label{SystemOfEquations}
	\left\{
	\begin{array}{ll}
		\displaystyle -\frac{d}{dt}\big(\varphi (t,v(t),|\dot{u}(t)|^{p-1})|\dot{u}(t)|^{p-2}\dot{u}(t)\big) = f(t,u(t),v(t)) & \text{for }t\in (0,1), \\
		 & \\
		\displaystyle -\frac{d}{dt}\big(|\dot{v}(t)|^{q-2}\dot{v}(t)\big)=g(t,u(t),v(t)) & \text{for }t\in (0,1), \\
		 & \\
		u(0) = u(1) = 0, & \\
		 & \\
		\displaystyle v(0)=\int_0^1h_0(v(s))dA_0(s), \quad v(1)=\int_0^1 h_1(v(s))dA_1(s). &
	\end{array}
	\right.
\end{equation}
Notice that, due to the already presented existence results, system \eqref{SystemOfEquations} is handled our hybrid method which glues together monotone operators and the fixed point approach. Therefore derivatives in the first equation is understood in the weak sense, while in the second one -- in a classical sense.
\begin{theorem}
    \label{TheoremSystemDifferential}
	Assume that Assumptions \ref{AssumptionPhi}, \ref{AssumptionF} and \ref{AssumptionG} hold. If there exists 
        \begin{equation*}
            \sigma < \frac{(p-1)(q-1)}{r}
        \end{equation*}
        and $\alpha,\beta > 0$ such that 
	\begin{equation}
	    \label{InequalitySigma}
		\frac{\delta(y)}{m(y)} \leq \alpha y^\sigma + \beta \quad \text{for every }y\geq 0,
	\end{equation}
	then system \eqref{SystemOfEquations} has at least one solution.
\end{theorem}
\begin{proof}
	We define $F\colon W^{1,p}_0(0,1)\times C[0,1]\longrightarrow W^{-1,p'}(0,1)$ and $G\colon W^{1,p}_0(0,1) \times C[0,1]\longrightarrow C[0,1]$ by
	\begin{equation*}
		\begin{split}
			\langle F(u,v), w\rangle = & \int_0^1 \varphi (t,v(t),|\dot{u}(t)|^{p-1})|\dot{u}(t)|^{p-2}\dot{u}(t) \dot{w}(t)dt- \int_0^1 f(t,u(t),v(t))w(t)dt,\\
			G(u,v)(t) = & \int_0^{t}\psi_q^{-1}\left(c(u,v)-\int_0^s g(\tau,u(\tau),v(\tau))d\tau\right)ds+\int_0^1h_0(v(s))dA_0(s).
		\end{split}
	\end{equation*}
	By Lemma \ref{LemmaOperatorDpPhi} operator $F$ satisfies Assumptions \ref{AssumptionH2}. To show that $G$ satisfies Assumption \ref{AssumptionH3} it is sufficient to apply Lemma \ref{Com} and recall the compactness of embedding $W^{1,p}_0(0,1)\hookrightarrow C[0,1]$. Indeed, take $u_n\rightharpoonup u_0$ in $W^{1,p}_0(0,1)$ and $v_n\to v_0$ in $C[0,1]$. Then both $u_n \to u_0$ and $v_n \to v_0$ in $C[0,1]$. Next we argue as in the proof of Lemma \ref{Com}. Therefore to apply Theorem \ref{TheoremExtended} it is sufficient to check assumption \eqref{StrangeAssumption}. Using \eqref{BoundednessOfT} and the Sobolev inequality we see there exists $a \in (0,1)$ and $b,c > 0$ such that
	\begin{equation*}
		\psi(x,y) = a y+ b x^{\frac{r}{q-1}} + c y^{\frac{\theta}{q-1}}.
	\end{equation*}
	Now, if $\gamma(x,y)\leq 0$, then $x^{p-1} \leq \frac{\delta(y)}{\lambda_p m(y)}$. Hence by assumption, there is $\widetilde{b},\widetilde{d} > 0$ such that
	\begin{equation*}
		\psi(x,y) \leq a y + \widetilde{b}y^{\frac{\sigma r}{(p-1)(q-1)}}y + cy^\frac{\theta}{q-1} + \widetilde{d}.
	\end{equation*}
	It is clear that there exists $R > 0$ such that 
	\begin{equation*}
		a y + \widetilde{b}y^{\frac{\sigma r}{(p-1)(q-1)}} + cy^\frac{\theta}{q-1} + \widetilde{d} \leq R \quad \text{whenever}\quad y \leq R.
	\end{equation*}
	Therefore Theorem \ref{TheoremExtended} can be applied to obtain the assertion.
\end{proof}

\begin{example}
    Let us consider a system 
    \begin{equation}
	    \label{SystemOfExample}
	    \left\{
	    \begin{array}{ll}
		    \displaystyle -\frac{d}{dt}\big(|\dot{u}(t)|\dot{u}(t)\big) = |v(t)|^2 - v(t)^2 u(t)^5 - v(t)^4u(t) + t^2 & \text{for }t\in (0,1), \\
		     & \\
		    \displaystyle -\frac{d}{dt}\big(|\dot{v}(t)|^{2}\dot{v}(t)\big) = v(t)\cos(v(t)) + u(t)\sqrt{|u(t)|} + \cos(u(t)) + v(t)\sin(t) & \text{for }t\in (0,1), \\
		     & \\
		    u(0) = u(1) = 0, & \\
		     & \\
		    \displaystyle v(0) = \int_0^1 \sin\big(v(s)\big) dA_0(s), \quad v(1)=\int_0^1 \cos\big(v(s)\big)dA_1(s), & 
	    \end{array}
	    \right.
    \end{equation}
    where $A_0,A_1\colon[0,1]\longrightarrow\mathbb{R}$ are arbitrary functions with a bounded variation.
    To apply Theorem \ref{TheoremSystemDifferential} we let $p = 3$, $q = 4$ and define $\varphi,f,g\colon [0,1]\times\mathbb{R}\times \mathbb{R}\longrightarrow \mathbb{R}$ by
    \begin{equation*}
        \begin{split}
            \varphi(t,u,v) & = 1,\\
            f(t,u,v) & = |v|^2 - v^2 u^5 - v^4u + t^2,\\
            g(t,u,v) & = v\cos(v) + u\sqrt{|u|} + \cos(u) + v\sin(t).
        \end{split}
    \end{equation*}
    We show that $\varphi$, $f$ and $g$ satisfy Assumption \ref{AssumptionPhi}, \ref{AssumptionF} and \ref{AssumptionG}, respectively. It is clear that conditions ($\Phi$1), ($\Phi$2), ($\Phi$3), ($\Phi$4), (F1), (F2), (F3), (G0) hold. Since (G2) is satisfied with $\alpha_0 = \alpha_1 = 0$ and $\beta_0 = \beta_1 = 0$, assumption (G3) holds. Moreover $\delta\colon [0,\infty) \longrightarrow [0,\infty)$ given by 
    \begin{equation*}
        \delta(v) = v^2 + 1
    \end{equation*}
    satisfies conditions (F4). Moreover
    \begin{equation*}
        |g(t,u,v)| \leq 2|v| + |u|^{3/2} + 1
    \end{equation*}
    Therefore we can take $B = 2$, $A = C = 1$, $r = \frac{3}{2}$ and $\theta = 1$ in condition (G1). Finally let us observe that  \eqref{InequalitySigma} holds with $\sigma = 2$, $\alpha = \beta = 1$. Therefore solvability of system \eqref{SystemOfExample} follows by Theorem \ref{TheoremSystemDifferential}.
\end{example}

\end{document}